\documentclass[10pt, reqno]{amsart}
\usepackage{amsaddr}
\usepackage{latexsym,amsfonts}
\usepackage{amssymb,amsthm,upref,amscd}
\usepackage[T1]{fontenc}
\usepackage{times}
\usepackage{amscd}
\usepackage{enumerate}
\usepackage{mathrsfs}
\usepackage{amsmath}
\usepackage{color}
\usepackage{soul}
\usepackage{indentfirst}
\usepackage{comment}
\PassOptionsToPackage{reqno}{amsmath}

\newtheorem{thm}{Theorem}[section]

\newtheorem{lem}{Lemma}[section]

\newtheorem{rem}{Remark}[section]
\theoremstyle{notation}

\newcommand{\R}{\mathbb{R}}
\newcommand{\C}{\mathbb{C}}
\numberwithin{equation}{section}

\makeatletter \@addtoreset{equation}{section} \makeatother

\usepackage[textwidth=138mm,
textheight=215mm,
left=30mm,
right=30mm,
top=25.4mm,
bottom=25.4mm,
headheight=2.17cm,
headsep=4mm,
footskip=12mm,
heightrounded]{geometry}

\usepackage[colorlinks=true,urlcolor=blue,
citecolor=black,linkcolor=black,linktocpage,pdfpagelabels,
bookmarksnumbered,bookmarksopen]{hyperref}
\newcounter{const}
\author[T. Gou]{Tianxiang Gou}
\author[{}]{{}}
\address[T. Gou]{%
\centerline{School of Mathematics and Statistics,Xi’an Jiaotong University,}
\centerline{710049, Xi’an, Shaanxi, China}}



\email{tianxiang.gou@xjtu.edu.cn}
\title[Blowup for 4NLS]{Blowup of cylindrically symmetric solutions for biharmonic NLS}
\thanks{Statements and Declarations. The author declares that there are no conflict of interests.}
\thanks{The author was supported by the National Natural Science Foundation of China (No. 12101483) and the Postdoctoral Science Foundation of China.}

\begin{document}

\begin{abstract}
In this paper, we consider blowup of solutions to the Cauchy problem for the following biharmonic nonlinear Schr\"odinger equation (NLS),
$$
\textnormal{i} \, \partial_t u=\Delta^2 u-\mu \Delta u-|u|^{2 \sigma} u \quad \text{in} \,\, \R \times \R^d,
$$
where $d \geq 1$, $\mu \in \R$ and $0<\sigma<\infty$ if $1 \leq d \leq 4$ and $0<\sigma<4/(d-4)$ if $d \geq 5$. In the mass critical and supercritical cases, we establish the existence of blowup solutions to the problem for cylindrically symmetric data. The result extends the known ones with respect to blowup of solutions to the problem for radially symmetric data.
\end{abstract}
\maketitle
\thispagestyle{empty}

\section{Introduction}

In this paper, we are concerned with blowup of cylindrically symmetric solutions to the Cauchy problem for the following biharmonic NLS,
\begin{align} \label{equ}
\textnormal{i} \, \partial_t u=\Delta^2 u-\mu \Delta u-|u|^{2 \sigma} u \quad \text{in} \,\, \R \times \R^d,
\end{align}
where $d \geq 1$, $\mu \in \R$ and $0<\sigma<\infty$ if $1 \leq d \leq 4$ and $0<\sigma<4/(d-4)$ if $d \geq 5$. The first study of biharmonic NLS traces back to Karpman \cite{K} and Karpman-Shagalov \cite{KS}, where the authors investigated the regularization and stabilization effect of the fourth-order dispersion. Later, Fibich et al. \cite{FIP} carried out a rigorous survey to biharmonic NLS from mathematical point of views and proved global existence in time of solutions to the Cauchy problem for \eqref{equ}. During recent years, there is a large number of literature mainly devoted to the study of well-posedness and scattering of solutions to the Cauchy problem for \eqref{equ}, see for example \cite{D, G, MXZ, P, P1, P2, PX} and references therein. In \cite{BL}, Boulenger and Lenzmann rigorously and completely discussed the existence of blowup solutions to the Cauchy problem for \eqref{equ} with radially symmetric initial data, which in turn confirms a series of numerical studies conducted in \cite{BFi, BFM1, BFM2, BF}. We also refer to the readers to the papers due to Bonheure et al. \cite{BF, BCGJ1, BCGJ2} with respect to orbital instability of radially symmetric standing waves to \eqref{equ}. Inspired by the aforementioned works, the aim of the present paper is to investigate blowup of solutions to the Cauchy problem for \eqref{equ} with cylindrically symmetric initial data, i.e. initial data belong to $\Sigma_d$ defined by
$$
\Sigma_d :=\left\{ u \in H^2(\R^d) : u(y, x_d) =u(|y|, x_d), x_d u \in L^2(\R^d)\right\},
$$
where $x=(y, x_d) \in \R^d$ and $y=(x_1, \cdots, x_{d-1}) \in \R^{d-1}$. 

For further clarifications, we shall fix some notations. Let us define 
$$
s_c:=\frac d 2 -\frac {2}{\sigma}.
$$
We refer to the cases $s_c<0$, $s_c=0$ and $s_c>0$ as mass subcritical, critical and supercritical, respectively. The end case $s_c=2$ is energy critical. Note that the cases $s_c=0$ and $s_c=2$ correspond to the exponents $\sigma=4/d$ and $\sigma=4/(d-4)$, respectively. For $1 \leq p <\infty$, we denote by $L^q(\R^{d})$ the usual Lebesgue space with the norm
$$
\|u\|_p:=\left( \int_{\R^d} |u|^p \,dx \right)^{\frac 1p}.
$$
The Sobolev space $H^2(\R^d)$ is equipped with the standard norm
$$
\|u\|:=\|\Delta u\|_2 + \|\nabla u\|_2 +\|u\|_2.
$$
In addition, we denote by $Q \in H^2(\R^d)$ a ground state to the following nonlinear elliptic equation,
$$
\Delta^2 Q +Q -|Q|^{2\sigma} Q=0 \quad \text{in} \,\, \R^d.
$$

The main results of the present paper read as follows, which gives blowup criteria for solutions to the Cauchy problem for \eqref{equ} with cylindrically symmetric data. 

\begin{thm} \label{thm1}\textnormal{(Blowup for Mass-Supercritical Case)}
Let $d \geq 5$, $\mu \in \R$ and $0<s_c<2$ with $0<\sigma \leq 1$. Suppose that $u_0 \in \Sigma_d$ satisfies one of the following conditions.
\begin{itemize}
\item [$(\textnormal{i})$] If $\mu \neq 0$, we assume that
\begin{align*}
E[u_0] < \left\{
\begin{aligned}
&0, \quad &\text{for} \,\, \mu>0, \\
&-\chi \mu^2 M[u_0], \quad &\text{for} \,\, \mu<0,
\end{aligned}
\right.
\end{align*}
with some constant $\chi=\chi(d, \sigma)>0$.
\item [$(\textnormal{ii})$] If $\mu=0$, we assume that either $E[u_0]<0$ or if $E[u_0] \geq 0$, we suppose that
$$
E[u_0]^{s_c}M[u_0]^{2-s_c}<E[Q]^{s_c}M[Q]^{2-s_c}
$$
and
$$
\|\Delta u_0\|_2^2 \|u_0\|_2^{2-s_c}>\|\Delta Q\|_2^2 \|Q\|_2^{2-s_c}.
$$
Then the solution $u \in C([0, T), H^2(\R^d))$ to the Cauchy problem for \eqref{equ} with initial datum $u_0$ blows up in finite time, i.e. $0<T<+\infty$ and $\lim_{t \to T^-} \|\Delta u\|_2=+\infty$.
\end{itemize}
\end{thm}

\begin{rem}
The extra restriction on $\sigma$ comes from the use of the well-known radial Sobolev inequality in $\R^{d-1}$. Note that if $4/d<\sigma \leq 1$, then $d \geq 5$. This is the reason that we need to assume that $d \geq 5$. 
\end{rem}

\begin{thm} \label{thm2} \textnormal{(Blowup for Mass-Critical Case)}
Let $d \geq 4$, $\mu \geq 0$ and $s_c=0$. Let $u_0 \in \Sigma_d$ be such that $E(u_0)<0$. Then the solution $u \in C([0, T), H^2(\R^d))$ to the Cauchy problem for \eqref{equ} satisfies the following.
\begin{itemize}
\item[$(\textnormal{i})$] If $\mu>0$, then $u(t)$ blows up in finite time.
\item [$(\textnormal{ii})$] If $\mu=0$, then $u(t)$ either blows up in finite time or $u(t)$ blows up in infinite time. 
\end{itemize}
\end{thm}

To prove Theorems \ref{thm1} and \ref{thm2}, the essential argument is to deduce the evolution of the localized virial quantity $M_{\varphi_R}[u(t)]$ defined by \eqref{defm} along time, see Lemma \ref{v}. To this end, we shall make use of  ideas from \cite{BL, M}. It is worth mentioning \cite{BF, BFG, DF, F}, where blow-up of solutions to NLS for cylindrically symmetric data has been investigated. Comparing with the existing works, we deal with the evolution of the localized virial quantity to biharmonic NLS for cylindrically symmetric data and extra treatments are needed in the cylindrically symmetric context, because of the presence of the biharmonic term. For cylindrically symmetric solutions, the radial Sobolev inequality is only applicable in $\R^{d-1}$, which is different from the radially symmetric case handled in \cite{BL}, we shall take advantage of ingredients in \cite{M} to estimate error terms due to the nonlinearity in the process of discussion of the evolution of the localized virial quantity.

\begin{rem} 
It seems possible to remove the condition that $x_d u_0 \in L^2(\R^d)$ to study blowup of solutions to \eqref{equ} for cylindrically symmetric data in the spirit of work due to Martel \cite{M}. In this case, more restrictive conditions should be imposed on $\sigma$. This shall be discussed in forthcoming publications.
\end{rem}

\section{Proofs of main results}

In this section, we are going to prove Theorems \ref{thm1} and \ref{thm2}. To do this, we first need to introduce a localized virial quantity, which is inspired by \cite{BL} and \cite{M}. For $d \geq 2$, let $\psi : \R^{d-1} \to \R$ be a radially symmetric and smooth function such that $|\nabla \psi^j| \in L^{\infty}(\R^{d-1})$ for $1 \leq j \leq 6$ and
\begin{align*}
\psi(r):=\left\{
\begin{aligned}
&\frac{r^2}{2} &\quad \text{for} \,\, r \leq 1,\\
&const. &\quad \text{for} \,\, r \geq 10,
\end{aligned}
\right.
\quad \psi''(r) \leq 1 \,\, \text{for}\,\, r \geq 0.
\end{align*}
For $R>0$ given, we define a radial function $\psi_R : \R^{d-1} \to \R$ by
$$
\psi_R(r):= R^2 \varphi\left(\frac{r}{R}\right).
$$
It follows from $(3.3)$ in \cite{BL} that
\begin{align} \label{property}
1-\psi_R''(r) \geq 0, \quad 1-\frac{\psi'_R(r)}{r} \geq 0, \quad d-1-\Delta \psi_R(r) \geq 0, \quad r \geq 0.
\end{align}
Let 
$$
\varphi_R(x):=\psi_R(r) + \frac{x_d^2}{2}, \quad  x=(y, x_d) \in \R^{d-1} \times \R, \,\, r=|y|.
$$
Define
\begin{align} \label{defm}
\mathcal{M}_{\varphi_R}[u]:=2 \textnormal{Im} \int_{\R^d} \overline{u} \left(\nabla \varphi_R \cdot \nabla u \right) \,dx.
\end{align}
It is simple see that $\mathcal{M}_{\varphi_R}[u]$ is well-defined for any $u \in \Sigma_d$. For later use, we shall give the well-known radial Sobolev's inequality in \cite{St}. For every radial function $f \in H^1(\R^{d-1})$ with $d \geq 3$, then
\begin{align} \label{st}
|y|^{\frac{d-2}{2}} |f(y)| \leq 2 \|f\|_{2}^{\frac 12} \|\nabla f\|_{2}^{\frac 12}, \quad  y\neq 0.
\end{align}
We also present the well-known Gagliardo-Nirenberg's inequality in one dimension. For any $f \in H^1(\R)$ and $p>2$, then
\begin{align} \label{gn}
\|f\|_{p} \leq C_p \|f'\|_2^{\alpha} \|f\|_2^{1-\alpha}, \quad \alpha=\frac{p-2}{2p}.
\end{align}
Let $f : \R^{d-1} \to \C$ be a radial and smooth function, then 
\begin{align}\label{d}
\partial_{kl}^2 f=\left(\delta_{kl}-\frac{x_kx_l}{r^2}\right) \frac{\partial_r f}{r} +\frac{x_kx_l}{r^2} \partial^2_rf.
\end{align} 
Next we present the well-posedness of solutions to the Cauchy problem for \eqref{equ} in $H^2(\R^d)$, which was established by Pausader \cite{P}.

\begin{lem} \cite[Proposition 4.1]{P}
Let $d \geq 1$, $\mu \in \R$ and $s_c<2$. Then, for any $u_0 \in H^2(\R^d)$, there exist a constant $T>0$ and a unique solution $u \in C([0, T), H^2(\R^d))$ to the Cauchy problem for \eqref{equ} with initial datum $u_0$. The solution has conserved mass and energy in the sense that
$$
M[u(t)]=M[u_0], \quad E[u(t)]=E[u_0], \quad t \in [0, T),
$$
where
$$
M[u]:=\int_{\R^d} |u|^2 \, dx
$$
and
$$
E[u]=\frac 12 \int_{\R^d} |\Delta u|^2 \, dx + \frac{\mu}{2} \int_{\R^d} |\nabla u|^2 \,dx -\frac{1}{2 \sigma +2} \int_{\R^d} |u|^{2 \sigma +2} \, dx.
$$
Moreover, blowup alternative holds, i.e. either $T = + \infty $ or $\|u(t)\|_{H^2}= +\infty$ as $t\to T^-$. The solution map 
$$
u_0 \in H^2(\R^d) \mapsto u \in C([0, T), H^2(\R^d))
$$ 
is continuous.
\end{lem}

In the following, we give the evolution of $\mathcal{M}[u(t)]$ along time, which is the key argument to prove Theorems \ref{thm1} and \ref{thm2}.

\begin{lem} \label{v}
Let $d \geq 3$, $R>0$ and $0<\sigma \leq 1$. Suppose that $u \in C([0, T); H^2(\R^d))$ is the solution to the Cauchy problem for \eqref{equ} with initial datum $u_0 \in \Sigma_d$. Then, for any $t \in [0, T)$, there holds that
\begin{align*}
\frac{d}{dt}\mathcal{M}_R[u(t)] &\leq 8 \int_{\R^d} |\Delta u|^2 \,dx+4 \mu \int_{\R^d} \left|\nabla u\right|^2 \,dx -\frac{2 \sigma d}{\sigma+1} \int_{\R^d}  |u|^{2 \sigma +2} \,dx + X_{\mu}[u] \\
& \quad + \mathcal{O}\left(R^{-4} +R^{-2} \|\nabla u\|_2^2 + R^{-\sigma(d-2)} \|\nabla u\|_2^{2\sigma} + |\mu|R^{-2}\right) \\
&=4d\sigma E(u_0)-(2d \sigma -8)\|\Delta u\|^2_2 -\mu (2d \sigma -4)\|\nabla u\|^2_2 + X_{\mu}[u]\\
& \quad +\mathcal{O}\left(R^{-4} +R^{-2} \|\nabla u\|_2^2 + R^{-\sigma(d-2)} \|\nabla u\|_2^{2\sigma} + |\mu|R^{-2}\right),
\end{align*}
where
\begin{align*}
X_{\mu}[u]\lesssim \left\{
\begin{aligned}
&0, &\quad \text{for} \,\, \mu \geq 0,\\
&|\mu| \|\nabla u\|_2^2 , &\quad \text{for} \,\, \mu<0.
\end{aligned}
\right.
\end{align*}
\end{lem}
\begin{proof}
To achieve this, we shall adapt some elements from \cite{BL} and \cite{M}. In view of Step 1 of the proof of \cite[Lemma 3.1]{BL}, we first have that
\begin{align*}
\frac{d}{dt}\mathcal{M}_R[u(t)]&=8\sum_{k,l,m=1}^d\left\langle u, \partial_{kl}^2\left(\partial_{lm}^2\varphi_R \right) \partial_{mk}^2 u \right\rangle +4\sum_{k,l=1}^d\left\langle u, \partial_k \left(\partial_{kl}\Delta \varphi_R \right) \partial_l u\right\rangle \\
& \quad +2\sum_{k=1}^d\left\langle u, \partial_k \left(\Delta^2 \varphi_R \right) \partial_l u\right\rangle +\left\langle u, \left(\Delta^3 \varphi_R \right) u\right\rangle  \\
& \quad  -4\mu\sum_{k,l=1}^d\left\langle u, \partial_k \left(\partial_{kl}^2\varphi_R \right) \partial_l u\right\rangle -\mu\left\langle u, \left(\Delta^2 \varphi_R \right) u\right\rangle -\left \langle u, [|u|^{2 \sigma}, \nabla \varphi_R \cdot \nabla + \nabla \cdot \nabla \varphi_R] u\right\rangle,\\
&=: \mathcal{A}_R^{(1)}[u] +\mathcal{A}_R^{(2)}[u] +\mathcal{B}_R[u].
\end{align*}
In what follows, we are going to estimate the terms $\mathcal{A}_R^{(1)}[u]$, $\mathcal{A}_R^{(2)}[u]$ and $\mathcal{B}_R[u]$. The estimates of dispersive terms $\mathcal{A}_R^{(1)}[u]$ and $\mathcal{A}_R^{(2)}[u]$ are inspired by the proof of \cite[Lemma 3.1]{BL}. Let us begin with treating the term $ \mathcal{A}_R^{(1)}[u]$. Using integration by parts and the definition of $\varphi_R$, we are able to derive that 
\begin{align}\label{a1}
\begin{split}
\sum_{k,l,m=1}^d\left\langle u, \partial_{kl}^2\left(\partial_{lm}^2\varphi_R \right) \partial_{mk}^2 u \right\rangle 
&=\sum_{k,l,m=1}^d \int_{\R^d}\left(\partial_{kl}^2 \overline{u}\right)\left(\partial_{lm}^2\varphi_R \right) \left(\partial_{mk}^2 u \right) \,dx \\
&=\sum_{k,l,m=1}^{d-1} \int_{\R^d}\left(\partial_{kl}^2 \overline{u}\right)\left(\partial_{lm}^2\psi_R \right) \left(\partial_{mk}^2 u \right) \,dx \\
& \quad + \sum_{l,m=1}^{d-1}\int_{\R^d}\left(\partial_{dl}^2 \overline{u}\right)\left(\partial_{lm}^2\psi_R \right) \left(\partial_{md}^2 u \right) \,dx \\
& \quad + \sum_{k=1}^{d-1} \int_{\R^d}\left(\partial_{kd}^2 \overline{u}\right)\left(\partial_{dk}^2 u \right) \,dx+ \int_{\R^d}\left|\partial_{dd}^2 u\right|^2 \, dx.
\end{split}
\end{align}
We now compute each term in the right hand side of \eqref{a1}. Utilizing \eqref{d}, we can derive that
\begin{align} \label{a11}
\begin{split}
&\sum_{k,l,m=1}^{d-1} \int_{\R^d}\left(\partial_{kl}^2 \overline{u}\right)\left(\partial_{lm}^2\psi_R \right) \left(\partial_{mk}^2 u \right) \,dx  \\
&=\sum_{k,l,m=1}^{d-1} \int_{\R}\int_{\R^{d-1}}\left(\partial_{kl}^2 \overline{u}\right)\left(\partial_{lm}^2\psi_R \right) \left(\partial_{mk}^2 u \right) \,dy dx_d \\
&=\int_{\R}\int_{\R^{d-1}} \left(\partial_r^2 \psi_R |\partial_r^2 u|^2 + \frac{d-2}{r^2} \frac{\partial_r \psi_R}{r} |\partial_r u|^2 \right)\, dydx_d \\
&=\int_{\R^d} |\Delta_y u|^2-\left(1-\partial_r^2 \psi_R\right)|\partial_r^2 u|^2 -\left(1-\frac{\partial_r \psi_R}{r}\right)\frac{d-2}{r^2}|\partial_r u|^2\,dx
\end{split}
\end{align}
and
\begin{align} \label{a12}
\begin{split}
\sum_{l,m=1}^{d-1}\int_{\R^d}\left(\partial_{dl}^2 \overline{u}\right)\left(\partial_{lm}^2\psi_R \right) \left(\partial_{md}^2 u \right) \,dx &= \sum_{k=1}^{d-1}\int_{\R^d}\left(\partial_{r}^2\psi_R \right) \left|\partial_{kd}^2 u \right|^2 \,dx \\
&=\sum_{k=1}^{d-1}\int_{\R^d}\left|\partial_{kd}^2 u \right|^2-\left(1-\partial_r^2 \psi_R\right)\left|\partial_{kd}^2 u \right|^2\,dx\\
&= \sum_{k=1}^{d-1} \int_{\R^d}\left(\partial_{kd}^2 \overline{u}\right)\left(\partial_{dk}^2 u \right) -\left(1-\partial_r^2 \psi_R\right)\left|\partial_{kd}^2 u \right|^2\,dx.
\end{split}
\end{align}
In addition, applying integration by parts, we have that
\begin{align} \label{i1}
\int_{\R^d}\left(\partial_{kd}^2 \overline{u}\right)\left(\partial_{dk}^2 u \right) \, dx =\int_{\R^d}\left(\partial_{kk}^2 \overline{u}\right)\left(\partial_{dd}^2 u \right) \, dx.
\end{align}
As a consequence, coming back to \eqref{a1} and using \eqref{property}, \eqref{a1}, \eqref{a11}, \eqref{a12} and \eqref{i1}, we now conclude that
\begin{align} \label{a2}
\sum_{k,l,m=1}^d\left\langle u, \partial_{kl}^2\left(\partial_{lm}^2\varphi_R \right) \partial_{mk}^2 u \right\rangle  \leq  \int_{\R^d} |\Delta u|^2 \,dx.
\end{align}
Furthermore, by the definitions of $\varphi_R$ and $\psi_R$, there holds that 
$$
\sum_{k,l=1}^d\left|\left\langle u, \partial_k \left(\partial_{kl}\Delta \varphi_R \right) \partial_l u\right\rangle \right| =\sum_{k,l=1}^{d-1}\left|\left\langle u, \partial_k \left(\partial_{kl}\Delta \psi_R\right) \partial_l u\right\rangle \right|\lesssim R^{-2} \|\nabla u\|_2^2, 
$$
$$
\sum_{k, l=1}^d\left|\left\langle u, \partial_k \left(\Delta^2 \varphi_R \right) \partial_l u\right\rangle \right| = \sum_{k, l=1}^{d-1}\left|\left\langle u, \partial_k \left(\Delta^2 \psi_R \right) \partial_l u\right\rangle \right|\lesssim R^{-2} \|\nabla u\|_2^2, 
$$
and
$$
\left|\left\langle u, \left(\Delta^3 \varphi_R \right) u\right\rangle \right| =\left|\left\langle u, \left(\Delta^3 \psi_R \right) u\right\rangle \right| \lesssim R^{-4} \|u\|_2^2.
$$
This along with \eqref{a2} and the conservation of mass implies that
$$
\mathcal{A}_R^{(1)}[u] \leq 8 \int_{\R^d} |\Delta u|^2 \,dx + \mathcal{O} \left(R^{-4} + R^{-2}\|\nabla u\|_2^2\right).
$$
We next deal with the term $\mathcal{A}_R^{(2)}[u]$. In virtue of integration by parts, the definition of $\varphi_R$ and \eqref{d}, we can show that
\begin{align*}
\mathcal{A}_R^{(2)}[u(t)]&=4 \mu \sum_{k, l=1}^d\int_{\R^d} \left(\partial_k \overline{u}\right) \left(\partial_{kl}^2 \varphi_R\right) \left(\partial_l u\right) \,dx -\mu \int_{\R^d} \left(\Delta^2 \varphi_R\right) |u|^2 \,dx \\
&=4 \mu \sum_{k, l=1}^{d-1}\int_{\R}\int_{\R^{d-1}} \left(\partial_k \overline{u}\right) \left(\partial_{kl}^2 \psi_R\right) \left(\partial_l u\right) \,dydx_d +4 \mu \int_{\R^d}\left |\partial_d u\right|^2 \,dx -\mu \int_{\R^d} \left(\Delta^2 \psi_R\right) |u|^2 \,dx  \\
&=4 \mu \int_{\R}\int_{\R^{d-1}} \left(\partial_r^2 \psi_R\right) \left|\partial_r u\right|^2 \,dydx_d +4 \mu \int_{\R^d}\left |\partial_d u\right|^2 \,dx -\mu \int_{\R^d} \left(\Delta^2 \psi_R\right) |u|^2 \,dx  \\
&=4 \mu \int_{\R^d} \left|\nabla_y u\right|^2 \,dx + X_{\mu}[u]+4 \mu \int_{\R^d}\left |\partial_d u\right|^2 \,dx-\mu \int_{\R^d} \left(\Delta^2 \psi_R\right) |u|^2 \,dx \\
&=4 \mu \int_{\R^d} \left|\nabla u\right|^2 \,dx + X_{\mu}[u]-\mu \int_{\R^d} \left(\Delta^2 \psi_R\right) |u|^2 \,dx
\end{align*}
where
$$
X_{\mu}[u]:=-4 \mu \int_{\R^d} \left(1-\partial_r^2 \psi_R\right) |\partial_r u|^2 \, dx.
$$
Due to $\|\Delta^2 \psi_R\|_{\infty} \lesssim R^{-2}$, then
\begin{align*}
\mathcal{A}_R^{(2)}[u(t)]&=4 \mu \int_{\R^d} \left|\nabla u\right|^2 \,dx + X_{\mu}[u]+\mathcal{O}(|\mu|R^{-2}).
\end{align*}
We now turn to handle the term $\mathcal{B}_{R}[u]$. Here we need some special treatments. Applying integration by parts and the definition of $\varphi_R$, we first derive that
\begin{align*} 
\mathcal{B}_{R}[u] =2 \int_{\R^d} |u|^2 \nabla \varphi_R \cdot \nabla \left(|u|^{2 \sigma}\right) &=-\frac{2 \sigma }{\sigma+1} \int_{\R^d}  \left(\Delta \varphi_R\right) |u|^{2 \sigma +2} \,dx \\
&=-\frac{2 \sigma }{\sigma+1} \int_{\R^d}  \left(\Delta \psi_R \right) |u|^{2 \sigma +2} \,dx- \frac{2 \sigma }{\sigma+1} \int_{\R^d} |u|^{2 \sigma +2} \,dx\\
&=-\frac{2 \sigma d}{\sigma+1} \int_{\R^d}  |u|^{2 \sigma +2} \,dx -\frac{2 \sigma }{\sigma+1} \int_{\R^d}\left(\Delta \psi_R -d+1\right) |u|^{2 \sigma +2} \,dx.
\end{align*}
In virtue of the definition of $\psi_R$ and \eqref{d}, then there holds that $\Delta \psi_R(r) -d+1=0$ for $0 \leq r \leq R$. This further implies that
\begin{align} \label{b1} 
\mathcal{B}_{R}[u]=-\frac{2 \sigma d}{\sigma+1} \int_{\R^d}  |u|^{2 \sigma +2} \,dx -\frac{2 \sigma }{\sigma+1} \int_{\R}\int_{|y| \geq R}  \left(\Delta \psi_R -d+1\right) |u|^{2 \sigma +2} \,dydx_d.
\end{align}
In the following, we shall estimate the second term in the right hand side of \eqref{b1}. Observe first that
\begin{align} \label{b11}
\int_{\R} \int_{|y| \geq R} |u|^{2 \sigma +2} \,dydx_d  \leq \int_{\R} \|u\|^{2 \sigma}_{L^{\infty}(|y| \geq R)} \|u\|_{L^2_y}^2 \,dx_d.
\end{align}
To proceed the proof, we first consider the case that $\sigma=1$. In this case, by \eqref{st}, H\"older's inequality and the conservation of mass, then
\begin{align} \label{b111}
\begin{split}
\int_{\R}\|u\|^2_{L^{\infty}(|y| \geq R)} \, dx_d &\lesssim R^{-(d-2)}\int_{\R} \|u\|_{L^2_y} \|\nabla_y u\|_{L^2_y} \,dx_d \\
&\leq R^{-(d-2)} \|u\|_2 \|\nabla _y u\|_2 \lesssim R^{-(d-2)} \|\nabla _y u\|_2.
\end{split}
\end{align}
On the other hand, by H\"older's inequality and the conservation of mass, we know that
\begin{align} \label{ud}
\begin{split}
\|u\|^2_{L^{\infty}(\R, {L^2_y}(\R^{d-1}))}&=\sup_{x_d \in \R} \int_{\R^{d-1}} |u|^2 \, dy = \sup_{x_d \in \R}  \int_{\R^{d-1}}  \int_{-\infty}^{x_d} \partial_{d}\left(|u|^2\right) \, dydx_d \\
&= 2 \textnormal{Re}\sup_{x_d \in \R}  \int_{\R^{d-1}} \int_{-\infty}^{x_d} \overline{u}\left(\partial_{d} u \right)  \, dydx_d \leq  2 \|u\|_2 \|\partial_d u\|_2 \lesssim \|\partial_d u\|_2.
\end{split}
\end{align}
Consequently, going back to \eqref{b11} and using \eqref{b111} and \eqref{ud}, we derive that
\begin{align} \label{s1}
\int_{\R}\int_{|y| \geq R}  |u|^4 \,dydx_d \lesssim R^{-(d-2)} \|\nabla_y u\|_2 \|\partial_d u\|_2 \lesssim R^{-(d-2)} \|\nabla u\|_2^2.
\end{align}
We next consider the case that $0<\sigma<1$. In this case, from \eqref{b11} and H\"older's inequality, it follows that
\begin{align} \label{b112}
\begin{split}
\int_{\R}\int_{|y| \geq R}  |u|^{2 \sigma +2} \,dx \leq \left(\int_{\R} \|u\|_{L^{\infty}(|y| \geq R)}^2\, d x_d \right)^{\sigma} \left(\int_{\R}\|u\|^{\frac{2}{1-\sigma}}_{L^2_y} \, d x_d \right)^{1-\sigma}.
\end{split}
\end{align}
In view of \eqref{gn} and the conservation of mass, we get that
\begin{align} \label{b12}
\begin{split}
\int_{\R}\|u\|^{\frac{2}{1-\sigma}}_{L^2_y} \, d x_d &\leq \left(\int_{\R} \left|\partial_d \left(\|u\|_{L^2_y}\right) \right|^2\, dx_d\right)^{\frac{\sigma}{2(1-\sigma)}} \left(\int_{\R}\|u\|^2_{{L^2_y}} \, dx\right)^{\frac{2-\sigma}{2(1-\sigma)}} \\
&\lesssim  \left(\int_{\R} \left|\partial_d \left(\|u\|_{L^2_y}\right) \right|^2\, dx_d\right)^{\frac{\sigma}{2(1-\sigma)}}.
\end{split}
\end{align}
Furthermore, notice that
\begin{align*}
\left|\partial_d \left(\|u\|_{L^2_y}\right) \right|\|u\|_{L^2_y}=\frac 12 \left|\partial_d \left(\|u\|^2_{L^2_y}\right) \right|=\frac 12  \left| \textnormal{Re} \int_{\R^{d-1}} \overline{u} \left(\partial_{d} u\right) \, dy\right| \leq \frac 12 \|\partial_d u\|_{L^2_y}\|u\|_{L^2_y}.
\end{align*}
This means that
\begin{align}\label{de}
\left|\partial_d \left(\|u\|_{L^2_y}\right) \right| \leq \frac 12 \|\partial_d u\|_{L^2_y}.
\end{align}
As a result, via \eqref{b12}, we obtain that
\begin{align} \label{b113}
\int_{\R}\|u\|^{\frac{2}{1-\sigma}}_{L^2_y} \, d x_d \lesssim \left(\int_{\R}\|\partial_d u\|_{L^2_y}^2 \, dx_d \right)^{\frac{\sigma}{2(1-\sigma)}}=\|\partial_d u\|_2^{\frac{\sigma}{1-\sigma}}.
\end{align}
Using \eqref{b111} and \eqref{b113}, we then obtain from \eqref{b112} that
\begin{align} \label{s2}
\int_{\R}\int_{|y| \geq R}  |u|^{2 \sigma +2} \,dx \lesssim R^{-\sigma(d-2) } \|\nabla _y u\|_2^{\sigma} \|\partial_d u\|_2^{\sigma} \lesssim R^{-\sigma(d-2)}\|\nabla u\|_2^{2\sigma}.
\end{align}
To sum up, it then follows from \eqref{b1}, \eqref{s1} and \eqref{s2} that
\begin{align*}
\mathcal{B}_{R}[u] \lesssim -\frac{2 \sigma d}{\sigma+1} \int_{\R^d}  |u|^{2 \sigma +2} \,dx +R^{-\sigma(d-2)}\|\nabla u\|_2^{2\sigma}.
\end{align*}
Accordingly, applying the estimates to $\mathcal{A}_R^{(1)}[u]$, $\mathcal{A}_R^{(2)}[u]$ and $\mathcal{B}_{R}[u]$ and the conservation of energy, we finally derive that
\begin{align*}
\frac{d}{dt}\mathcal{M}_R[u(t)] & \leq 8 \int_{\R^d} |\Delta u|^2 \,dx+4 \mu \int_{\R^d} \left|\nabla u\right|^2 \,dx -\frac{2 \sigma d}{\sigma+1} \int_{\R^d}  |u|^{2 \sigma +2} \,dx + X_{\mu}[u] \\
& \quad + \mathcal{O}\left(R^{-4} +R^{-2} \|\nabla u\|_2^2 + R^{-\sigma(d-2)} \|\nabla u\|_2^{2\sigma} + |\mu|R^{-2}\right) \\
&=4d\sigma E(u_0)-(2d \sigma -8)\|\Delta u\|^2_2 -\mu (2d \sigma -4)\|\nabla u\|^2_2 + X_{\mu}[u]\\
& \quad +\mathcal{O}\left(R^{-4} +R^{-2} \|\nabla u\|_2^2 + R^{-\sigma(d-2)} \|\nabla u\|_2^{2\sigma} + |\mu|R^{-2}\right),
\end{align*}
This completes the proof.
\end{proof}

\begin{proof}[Proof of Theorem \ref{thm1}]
Noting that $0<\sigma \leq 1$ and applying Lemma \ref{v}, then the proof can be completed by closely following the one of \cite[Theorem 1]{BL}.
\end{proof}

\begin{proof}[Proof of Theorem \ref{thm2}]
If $\mu>0$, using Lemma \ref{v} and arguing as the proof of \cite[Theorem 3]{BL}, we can get the desired result. To complete the proof, we only need to consider the case that $\mu=0$. In this case, we need to conduct a more refined analysis to the evolution of $\mathcal{M}[u(t)]$ along time. From the proof of Lemma \ref{v}, \eqref{property} and \eqref{d}, we first have that
\begin{align} \label{vm}
\begin{split}
\frac{d}{dt}\mathcal{M}_R[u(t)] &\leq 16 E(u_0)-8\int_{\R^d}\left(1-\partial_r^2 \psi_R\right)|\partial_r^2 u|^2 \, dx + \int_{\R^d} \left(\Delta^3 \psi_R \right)|u|^2 \,dx \\
& \quad -\int_{\R^d} A_R |\partial_r u|^2 \, dx + \int_{\R^N} B_R |u|^{2+ \frac 8 d} \, dx ,
\end{split}
\end{align}
where
$$
A_R:=4 \partial_r^2 \Delta \psi_R +2 \Delta^2 \psi_R(r), \quad B_R:=\frac{8}{d+4} \left(d-1-\Delta \psi_R\right).
$$
Thereinafter, we shall estimate the last two terms in the right hand side of \eqref{vm}. Using $(7.7)$ in \cite{BL} and the conservation of mass, we can derive that
\begin{align*}
\left|\int_{\R^d} A_R |\partial_r u|^2 \, dx \right| =\left|\int_{\R}\int_{\R^{d-1}} A_R |\partial_r u|^2 \, dydx_d \right|
&\lesssim  8 \eta R^4 \int_{\R} \|A_R \partial_r^2 u\|_{L^2_y}^2  \, dx_d+ \eta^{-1}R^{-4} \int_{\R} \|u\|_{L^2_y}^2 \,dx_d \\
&=8 \eta R^4 \|A_R \partial_r^2 u\|_2^2 + \eta^{-1}R^{-4},
\end{align*}
where $\eta>0$ is an arbitrary constant. We now treat the other term. To do this, we first consider the case that $d=4$. In this case, we have that
\begin{align}\label{mc0}
\begin{split}
\left|\int_{\R^d} B_R |u|^4 \, dx\right|=\left|\int_{\R}\int_{|y| \geq R} B_R |u|^4\, dydx_d \right| &\leq \int_{\R} \|u\|^2_{L^2_y} \|B_R^{\frac 12} u\|_{L^{\infty}(|y| \geq R)}^2 \, dx_d \\
&\leq \|u\|^2_{L^{\infty}(R, {L^2_y}(\R^{N-1}))} \int_{\R} \|B_R^{\frac 12} u\|_{L^{\infty}(|y| \geq R)}^2 \, dx_d,
\end{split}
\end{align}
because of $B_R=0$ for $0 \leq r \leq R$. In view of \eqref{st} with $d=4$ and the definition of $B_R$, we are able to infer that
\begin{align} \label{mc1}
\begin{split}
\int_{\R}\|B_R^{\frac 12} u\|_{L^{\infty}(|y| \geq R)}^2 \, dx_d &\lesssim R^{-2}\int_{\R}\|B_R^{\frac 12} u\|_{L^2_y}\|\nabla_y (B_R^{\frac 12} u)\|_{L^2_y} \,dx_d \\
&\lesssim R^{-2}\int_{\R}\|u\|_{L^2_y}\|\nabla_y (B_R^{\frac 12 } u)\|_{L^2_y} \,dx_d.
\end{split}
\end{align}
It follows from $(7.9)$ with $d=4$ in \cite{BL} that 
$$
\|\nabla_y (B_R^{\frac 12} u)\|_{L^2_y}^2  \lesssim \left(\eta^{-\frac 14} + R^{-2}\right) \|u\|_{L^2_y}^2 + 8 \eta^{\frac 14} \|B_R \partial_r^2 u\|_{L^2_y}^2.
$$
By means of \eqref{mc1}, the conservation of mass and H\"older's inequality, we then get that
\begin{align} \label{mc22} 
\begin{split}
\int_{\R}\|B_R^{\frac 1 2} u\|_{L^{\infty}(|y| \geq R)}^{2} \, dx_d &\lesssim R^{-2} \left(\eta^{-\frac 14} + R^{-2}\right) + 8R^{-2} \eta^{\frac 14} \int_{\R^d} \|u\|_{L^2_y} \|B_R \partial_r^2 u\|_{L^2_y} \, dx_d \\
& \lesssim R^{-2} \left(\eta^{-\frac 14} + R^{-2}\right) + 8 R^{-2} \eta^{\frac 14} \|B_R \partial_r^2 u\|_2.
\end{split}
\end{align}
Going back to \eqref{mc0}  and using \eqref{ud} and \eqref{mc22}, we then get that
\begin{align*}
\left|\int_{\R^N} B_R |u|^4 \, dx \right| &\lesssim R^{-2} \left(\eta^{-\frac 14} + R^{-2}\right) \|\partial_d u\|_2 +8R^{-2} \eta^{\frac 14} \|B_R\partial_r^2 u\|_2 \|\partial_d u\|_2 \\
& \leq  R^{-2} \left(\eta^{-\frac 14} + R^{-2}\right) \|\partial_d u\|_2 +R^{-4} \eta^{-\frac 12} \|\partial_d u\|_2^2 + 8 \eta \|B_R \partial_r^2 u\|_2^2.
\end{align*}
Taking into account \eqref{vm} and noting that $\|\Delta^2 \varphi_R\|_{\infty} \lesssim R^{-2}$, we then obtain that
\begin{align} \label{m1}
\begin{split}
\frac{d}{dt}\mathcal{M}_R[u(t)] &\leq 16 E(u_0)-8\int_{\R^d}\left(1-\partial_r^2 \psi_R-\eta \left(R^4 A_R^2 +B_R^2\right)\right)|\partial_r^2 u|^2 \, dx \\
&\quad + R^{-2} \left(\eta^{-\frac 14} + R^{-2}\right) \|\partial_d u\|_2 +R^{-4} \eta^{-\frac 12} \|\partial_d u\|_2^2 +\eta^{-1}R^{-4} + R^{-2}.
\end{split}
\end{align}
Next we consider the case that $d \geq 5$. In this case, by H\"older's inequality and the definition of $B_R$, we can obtain that
\begin{align} \label{d5}
\begin{split}
\left|\int_{\R^N} B_R |u|^{2 +\frac 8 d} \, dx \right| &=\left|\int_{\R} \int_{|y| \geq R} B_R |u|^{2 +\frac 8 d} \, dydx_d \right|  
\leq \int_{\R} \|B_R^{\frac 12} u\|_{L^{\infty}_y(|y| \geq R)}^{\frac 8 d} \|B_R^{\frac 1 2-\frac 2d} u\|_{L^2_y}^2\, dx_d \\
& \leq \left(\int_{\R} \|B_R^{\frac 12} u\|_{L^{\infty}_y(|y| \geq R)}^2 \, dx_d \right)^{\frac 4 d} \left(\int_{\R}\|u\|_{L^2_y}^{\frac{2d}{d-4}} \, dx_d\right)^{\frac{d-4}{d}}.
\end{split}
\end{align}
As an application of \eqref{gn} leads to
\begin{align} \label{sd}
\int_{\R}\|u\|_{L^2_y}^{\frac{2d}{d-4}} \, dx_d \lesssim \left(\int_{\R^d} \left|\partial_d \left(\|u\|_{L^2_y}\right)\right|^2\, dx_d\right)^{\frac{2}{d-4}} \left(\int_{\R} \|u\|_{L^2_y}^2 \,dx_d\right)^{\frac{d-2}{d-4}}.
\end{align}
It then follows from \eqref{de}, \eqref{sd} and the conservation of mass that
\begin{align} \label{mc11}
\int_{\R}\|u\|_{L^2_y}^{\frac{2d}{d-4}} \, dx_d \lesssim \left(\int_{\R^d} \|\partial_d u\|_{L^2_y}^2\, dx_d\right)^{\frac{2}{d-4}}=\|\partial_d u\|_2^{\frac{4}{d-4}}.
\end{align}
Therefore, coming back to \eqref{d5} and using \eqref{mc22} and \eqref{mc11}, we derive that
\begin{align*}
\left|\int_{\R^N} B_R |u|^{2 +\frac 8 d} \, dx \right|  &\lesssim \left(\int_{\R} \|B_R^{\frac 12} u\|_{L^{\infty}_y(|y| \geq R)}^2 \, dx_d\right)^{\frac 4 d}\|\partial_d u\|_2^{\frac 4 d} \\
& \lesssim R \int_{\R} \|B_R^{\frac 12} u\|_{L^{\infty}_y(|y| \geq R)}^2 \, dx_d + R^{-\frac{4}{d-4}} \|\partial_d u\|_2^{\frac{4}{d-4}} \\
& \leq R^{-1} \left(\eta^{-\frac 14} + R^{-2}\right)\|\partial_d u\|_2 + 8 R^{-1} \eta^{\frac 14} \|B_R \partial_r^2 u\|_2 + R^{-\frac{4}{d-4}} \|\partial_d u\|_2^{\frac{4}{d-4}}  \\
& \lesssim  R^{-1} \left(\eta^{-\frac 14} + R^{-2}\right)\|\partial_d u\|_2 + 8 \eta \|B_R \partial_r^2 u\|_2^2 + R^{-\frac{4}{d-4}} \|\partial_d u\|_2^{\frac{4}{d-4}} + R^{-2} \eta^{-\frac 12}.
\end{align*}
As a consequence, invoking \eqref{vm}, we finally derive that
\begin{align} \label{m2}
\begin{split}
\frac{d}{dt}\mathcal{M}_R[u(t)] &\leq 16 E(u_0)-8\int_{\R^d}\left(1-\partial_r^2 \psi_R-\eta \left(R^4 A_R^2 +B_R^2\right)\right)|\partial_r^2 u|^2 \, dx \\
&\quad + R^{-1} \left(\eta^{-\frac 14} + R^{-2}\right) \|\partial_d u\|_2 +R^{-4} \eta^{-\frac 12} \|\partial_d u\|_2^2 \\
& \quad+ R^{-\frac{4}{d-4}} \|\partial_d u\|_2^{\frac{4}{d-4}} + R^{-2} \eta^{-\frac 12}+\eta^{-1}R^{-4} + R^{-2}.
\end{split}
\end{align}
At this point, using \eqref{m1} and \eqref{m2} and reasoning as the proof of \cite[Theorem 3]{BL}, we are able to finish the proof. This completes the proof.
\end{proof}

\end{document}